\pdfoutput=1
\RequirePackage{ifpdf}
\ifpdf % We~are running pdfTeX in pdf mode
\documentclass[pdftex]{sigma}
\else
\documentclass{sigma}
\fi

\numberwithin{equation}{section}

\newtheorem{Theorem}{Theorem}[section]
\newtheorem{Corollary}[Theorem]{Corollary}
\newtheorem{Lemma}[Theorem]{Lemma}
\newtheorem{Proposition}[Theorem]{Proposition}
{ \theoremstyle{definition}
	\newtheorem{Definition}[Theorem]{Definition}
	
	\newtheorem{Example}[Theorem]{Example}
	\newtheorem{Remark}[Theorem]{Remark} }

\usepackage[all]{xy}

\def\NN{\ensuremath{\mathbb{N}}}

\def\RR{\ensuremath{\mathbb{R}}}
\def\ZZ{\ensuremath{\mathbb{Z}}}

\DeclareMathOperator{\id}{id}

\newcommand{\boundary}[1]{\partial#1}

\newcommand{\disjointunion}{\amalg}

\DeclareMathOperator{\Hom}{Hom} %Homomorphisms

\newcommand{\bigdisjointunion}{\bigsqcup}
\DeclareMathOperator{\Orb}{Orb}
\DeclareMathOperator{\GTOP}{G-TOP}
\newcommand{\forget}[1]{}

\begin{document}

\allowdisplaybreaks

\newcommand{\arXivNumber}{2412.07955}

\renewcommand{\PaperNumber}{093}

\FirstPageHeading

\ShortArticleName{Stolz Positive Scalar Curvature Structure Groups}

\ArticleName{Stolz Positive Scalar Curvature Structure Groups,\\ Proper Actions and Equivariant 2-Types}

\Author{Massimiliano PUGLISI~$^{\rm a}$, Thomas SCHICK~$^{\rm b}$ and Vito Felice ZENOBI~$^{\rm c}$}

\AuthorNameForHeading{T.~Schick, M.~Puglisi and V.F.~Zenobi}

\Address{$^{\rm a)}$~Dipartimento di Matematica, Sapienza Universit\`a di Roma, Italy}
\EmailD{\mail{puglisi1994@gmail.com}}

\Address{$^{\rm b)}$~Mathematisches Institut, Universit\"at G\"ottingen, Germany}
\EmailD{\mail{thomas.schick@math.uni-goettingen.de}}
\URLaddressD{\url{http://www.uni-math.gwdg.de/schick/}}

\Address{$^{\rm c)}$~Istituto Nazionale di Alta Matematica, Piazzale Aldo Moro 5, 00185 Roma, Italy}
\EmailD{\mail{zenobi@altamatematica.it}}

\ArticleDates{Received February 11, 2025, in final form October 19, 2025; Published online October 30, 2025}

\Abstract{In this note, we study equivariant versions of Stolz' $R$-groups, the positive scalar 	curvature structure groups $R^{\rm spin}_n(X)^G$, for proper actions of discrete groups~$G$. We define the concept of a fundamental groupoid functor for a $G$-space, encapsulating all the fundamental group information of all the fixed point sets and their relations. We construct classifying spaces for fundamental groupoid functors. As a geometric result, we show that Stolz' equivariant $R$-group $R^{\rm spin}_n(X)^G$ depends only on the fundamental groupoid functor of the reference space~$X$. The proof covers at the same time in a concise and clear way the classical non-equivariant case.}


\Classification{57R91; 57R90; 53C27; 53C21}

\section{Introduction}

In \cite{Stolz}, Stolz organized the classification problem of metrics with positive scalar curvature in a~long exact sequence
\begin{equation}
	\label{eq:Stolz}
	\cdots \to R^{\rm spin}_{n+1}(X) \to \mathrm{Pos}^{\rm spin}_n(X) \to \Omega^{\rm spin}_n(X) \to R^{\rm spin}_n(X) \to \cdots
\end{equation}
of cobordism groups, where $X$ is a CW-complex. This long exact sequence
includes the well-known spin cobordism $\Omega_*^{\rm spin}(X)$; then the
structure bordism group of metrics of positive scalar curvature
$\mathrm{Pos}_*^{\rm spin}(X)$, whose cycles $(f\colon M\to X,g)$ are defined like
those of $\Omega_*^{\rm spin}(X)$, but adding the geometric secondary structure
given by a positive scalar curvature metric $g$ on the smooth compact manifold
$M$ with spin structure; finally, the object of study of this note,
$ R^{\rm spin}_*(X)$ which intuitively arises as a sort of mapping cone construction of the forgetful
functor \[\mathrm{Pos}^{\rm spin}_*(X) \to \Omega^{\rm spin}_*(X).\]

An explicit calculation of this last group has not yet been achieved in any case of
interest. Nonetheless, Stolz has proved that, if $X=B\pi_1(M)$ is the
classifying space of the fundamental group of a closed spin manifold $M$ with
$\dim(M)\ge 5$, then
$ R^{\rm spin}_{\dim M+1}(X)$ acts freely and transitively on the space of
concordance classes of positive scalar curvature metrics on $M$. This action
is canonical and hence the space of concordance classes
inherits canonically (after the choice of a base point) an abelian group structure and
a construction of an explicit model for this is given in \cite[Section 4]{XieYuZeidler}. This has been used, for
instance, in \cite{PSZ2,PSZ,SZ,XieYuZeidler} to give a lower bound on the rank
of this affine group and of the moduli space of concordance classes of
metrics of positive scalar curvature, where moduli space refers to the quotient by the action of the
diffeomorphism group of $M$.

A fundamental step in \cite{SZ} uses the fact, which can be implicitly deduced
from \cite{Stolz}, that a~2-connected map between CW-complexes, such as the
classifying map $u\colon M\to B\pi_1(M)$, induces an isomorphism between
corresponding
$ R^{\rm spin}_*$ groups for $*\ge 6$.

There is an obvious equivariant reformulation of this story: instead of
working with a space~$X$ one can work with the a $G$-cover $\bar X$ with its free
$G$-action. One can then go equivalently back and forth between structure on
$X$ and $G$-invariant structure on~$\bar X$ (pulling back from~$X$ to~$\bar X$
and quotienting by $G$ from~$\bar X$ to~$X$). This way, for a space $X$ with
fundamental group $G$ and universal cover $\tilde{X}$, one can replace equivalently the Stolz sequence
\eqref{eq:Stolz} by its equivariant version
\[%	\label{eq:Stolz_eq}
	\cdots \to R^{\rm spin}_{n+1}\bigl(\tilde X\bigr)^G \to
	\mathrm{Pos}^{\rm spin}_n\bigl(\tilde X\bigr)^G \to \Omega^{\rm spin}_n\bigl(\tilde X\bigr)^G
	\to R^{\rm spin}_n\bigl(\tilde X\bigr)^G \to \cdots,
\]
where all the cycles are defined as above, by requiring to come with a free and
co-compact and (whenever we have metrics) isometric action of $G$, and all maps
are required to be $G$-equivariant. This equivalent reformulation is important to get information about
the groups in the Stolz sequence via higher index theory, which has been
successfully implemented in \cite{PiazzaSchick} and~\cite{Zeidler}, compare
also \cite{XieYu} and the survey \cite{SchickICM}. The proofs of lower bounds
on the rank of the affine group
of concordance classes of positive scalar curvature are based on these techniques.

The aim of this note is two-fold. The first main contribution is an explicit
and
concise proof of the fact that (non-equivariantly) the group $R^{\rm spin}_*(X)$
depends for $*\ge 6$ only on the
fundamental group information of $X$. Secondly, we want to analyze the
corresponding statement for the case of general
proper $G$-actions. Here, we develop basic tools for this
generalization and then study it in the equivariant context of a CW-complex endowed with a proper
action of a discrete group. Concretely, we construct an equivariant
version of the Stolz exact sequence, compare Proposition~\ref{prop:Stolz_seq}.
One contribution is to give a complete proof of exactness, as we are not aware
that this is available in the literature. Our main original result is the fact
that the equivariant $R$-groups do only depend on the equivariant 2-type of the
space, compare Theorem~\ref{G2conn}. Along the way, we construct an
equivariant analog of the space $B\pi_1(X)$, namely a ``universal space for a
given 2-type'', compare Section \ref{sec:universal_space}.

It should be noted that this relies on surgery constructions requiring
enough ``room''. This is the reason for the dimension restrictions listed
above: the manifolds whose positive scalar curvature metrics are controlled
have to be spin manifolds of dimension at least $5$, and the group
$R^{\rm spin}_*(X)^G$ can be treated efficiently if $*\ge 6$.

\section{Proper actions of discrete groups}

Let us fix throughout a discrete group $G$.

\begin{Definition}
	A $G$-CW-complex X is a $G$-space together
	with a $G$-invariant filtration
	\[
\varnothing = X^{(-1)}\subseteq X^{(0)} \subseteq X^{(1)} \subseteq \dots \subseteq X^{(n)}\subseteq \dots \subseteq\bigcup_{n\geq0}X^{(n)}=X
\]
	such that $X$ carries the colimit topology with respect to this filtration and for each $n\geq0$
	the space $X^{(n)}$ is obtained from $X^{(n-1)}$ by attaching equivariant $n$-dimensional
	cells, i.e., there exists a $G$-pushout
	\[
	\xymatrix{\displaystyle\bigsqcup_{i\in I_n}G/H_i\times S^{n-1}\ar[r]\ar[d]& X^{(n-1)}\ar[d]\\
		\displaystyle\bigsqcup_{i\in I_n}G/H_i\times D^n\ar[r]& X^{(n)}.}
	\]
	
	Note that the $G$-CW-complex $X$ defines its \emph{isotropy family}
	$\mathcal{I}(X)$ of subgroups of $G$ where~$H$ belongs to $\mathcal{I}(X)$ if and
	only if the fixed point set $X^H$ is non-empty or in other words if and only
	if $X$ contains a $G$-cell $G/K\times D^n$ where a conjugate of $H$ is
	contained in $K$.
\end{Definition}

We recall the concept of a family of subgroups which we just have used.

\begin{Definition}
	Let $G$ be a discrete group. A \emph{family of subgroups $\mathcal{F}$} is a
	set of subgroup of $G$ which is closed under conjugation and is closed under passing
	to smaller subgroups.
	
	Significant examples of such families of subgroups are $\mathcal{FIN}$, the family of all finite subgroups, or
	$\mathcal{ALL}$, the family of all subgroups.
\end{Definition}

Let us fix some notation.
\begin{Definition}
	Let $f\colon X \rightarrow Y$ be a continuous $G$-equivariant map between
	$G$-spaces. We will denote by $f^H\colon X^H\to Y^H$ the restriction of $f$
	to the $H$-fixed point sets, with $H$ a subgroup of~$G$.
\end{Definition}

\begin{Definition}
	We say that $f$ is cellular if $X$ and $Y$ are $G$-CW-complexes and, denoting by~$X^{(k)}$ the \textit{k-skeleton} of X, one has $f\bigl(X^{(k)}\bigr) \subseteq Y^{(k)}$.
\end{Definition}
The well-known and important cellular approximation theorem extends to the equivariant context (compare \cite[Theorem~2.1]{ttd}).

\begin{Theorem} \label{cellular}
	Let $f\colon X\to Y$ be a $G$-map. Then there exists a G-homotopy $h\colon X\times I\to Y$ such that $h_0:= h_{X\times\{0\}}=f$ and $h_1:= h_{X\times\{1\}}$ is cellular.
\end{Theorem}

We also have an equivariant version of the Whitehead Theorem for $G$-CW-complexes.

\begin{Definition}\label{def:connectedness}
	Consider a function $\nu\colon \mathcal{ALL}\to\mathbb{N}$. Then we say that
	$f$ is~$\nu$-connected if~$f^H$ is~$\nu(H)$-connected for all $H\in
	\mathcal{F}$, namely the induced maps are isomorphisms on the first
	$\nu(H)-1$ homotopy groups of (all components of) $X^H$ and $Y^H$ and a
	surjection on the $\nu(H)$-th one.
	In particular, we say that it is {$k$-connected} if $\nu$ is constantly equal to $k$.
	
	Moreover, we say that a relative $G$-CW-complex $(X,A)$ has dimension less or equal to $\nu$ if the cells in $X\setminus A$ are of the form $G/H\times D^k$ with $k\leq\nu(H)$.
\end{Definition}
\begin{Proposition}[{compare \cite[Proposition~2.6]{ttd}}] \label{whitehead}
	Let $f\colon B\to C$ be a $\nu$-connected map between $G$-CW-complexes and $A$
	another
	$G$-CW-complex. Write $[A,B]^G$ for the set of $G$-homotopy classes of
	$G$-maps from $A$ to $B$. Then
	\begin{equation*}
		f_*\colon\ [A,B]^G\to [A,C]^G
	\end{equation*}
	is surjective $($or
	bijective, respectively$)$ if $\dim A\leq\nu$ $($or $\dim A<\nu$, respectively$)$.
\end{Proposition}

Note that the classical Whitehead theorem is a consequence, using $A=C$
and $A=B$ and identity maps.

\section{The Stolz exact sequence}
	
\begin{Definition}\label{def:Stolz_groups}
		Let $X$ be a $G$-CW-complex. We then define the following groups:		
		\begin{itemize}\itemsep=0pt
			\item $\Omega^{\rm spin}_{n}(X)^{G}$ is the $G$-equivariant \textit{spin
				bordism} group: a cycle here is given by a pair $(M,f)$, where $M$ is an
			$n$-dimensional spin-manifold with cocompact spin-structure
			preserving\footnote{That is, with a lift of the action to the Spin principal
				bundle.} action of $G$ and with a $G$-equivariant
			reference map $f\colon M\to X$. Two cycles $(M,f)$ and~$\bigl(M',f'\bigr)$ are
			equivalent if there is a cocompact spin $G$-bordism $W$ from $M$ to $M'$
			and there exists a $G$-equivariant reference map $F\colon W\to X$
			extending $f$ and $f'$.
			\item $\mathrm{Pos}^{\rm spin}_{n}(X)^{G}$ consists of cycles $(M,f,g)$, where the pair
			$(M,f)$ is as before and $g$ is a~$G$-invariant metric with positive
			scalar curvature on $M$. Two such cycles $(M,f,g)$ and~$\bigl(M',f',g'\bigr)$ are
			equivalent if there exists a~spin bordism $(W,F)$ as before, along with a~$G$-invariant metric $g_W$ on $W$ which is of product type near the
			boundary which restricts to $g$ on $M$ and to $g'$ on $M'$.
			\item $R^{\rm spin}_{n}(X)^{G}$ is the bordism group of spin $G$-manifolds with
			boundary (possibly empty) of dimension $n$, together with a $G$-invariant
			Riemannian metric of positive scalar curvature on the boundary. Bordisms are then manifolds with
			corners. In particular, $(M,f,g)$ and $\bigl(M',f',g'\bigr)$ are equivalent if there
			exists a $G$-bordism $(W,F,\bar{g})$, where $W$ is a bordism between $M$
			and $M'$ and the resulting bordism $\partial_0 W$ between $\partial M$ and
			$\partial M'$ carries a $G$-invariant
			metric $\bar{g}$ with positive scalar curvature, so that it is a
			bordism between~${\bigl(\partial M,\partial f,g\bigr)}$ and
			$\bigl(\partial M',\partial f',g'\bigr)$ in the sense of $\mathrm{Pos}^{\rm spin}_{n-1}(X)^{G}$.
		\end{itemize}
		
		Each of these sets is equipped with an abelian group structure given by
		disjoint union of manifolds and is covariantly functorial in $X$ as follows: a $G$-equivariant map of $G$-CW-Complexes~${\varphi\colon X \rightarrow Y}$ induces a mapping $\varphi_{*}$ on these groups
		by post-composing the reference maps with~$\varphi$.
	\end{Definition}
	
	\begin{Remark}
		Note that for each cycle $f\colon M\to X$ in $\Omega_n^{\rm spin}(X)^G$ the
		isotropy of $M$ is restricted to belong to the family $\mathcal{I}(X)$
		(because the image under the equivariant map $f$ of a point $x\in M$ fixed by a subgroup $H$ of
		$G$ must also be fixed by $H$, and is a point in $X$).
		
		If we want to restrict the isotropy even further, to live in a family
		$\mathcal{F}$ of subgroups of $G$, we can replace the space $X$ by $X\times
		E_{\mathcal{F}}G$ where $E_{\mathcal{F}}G$ is the universal $G$-CW-complex
		with isotropy family $\mathcal{F}$. This space is characterized by the property
		that $E_{\mathcal{F}}G^H$ is empty if $H\notin \mathcal{F}$ and~$E_{\mathcal{F}}G^H$ is contractible if $H\in\mathcal{F}$, in particular $\mathcal{I}(E_{\mathcal{F}}G)=\mathcal{F}$. It exists for
		each family $\mathcal{F}$ and it is unique up to $G$-equivariant homotopy
		equivalence. It has the universal property that a~$G$-CW-complex~$X$ with
		isotropy contained in $\mathcal{F}$ has a unique homotopy class of $G$-maps
		to~$E_{\mathcal{F}}G$. These spaces were introduced and studied in \cite{tomDieck,ttd}.
	\end{Remark}

\begin{Proposition}\label{prop:Stolz_seq}
		The abelian groups defined in Definition~{\rm\ref{def:Stolz_groups}} fit into the
		following $G$-equivari\-ant version of the Stolz positive scalar curvature exact sequence:
		\[
			\cdots \to R^{\rm spin}_{n+1}(X)^{G} \to \mathrm{Pos}^{\rm spin}_{n}(X)^{G} \to \Omega^{\rm spin}_{n}(X)^{G} \to R^{\rm spin}_{n}(X)^{G} \to \cdots,
		\]
		where the first map sends a manifold to its boundary, the second one
		is the forgetful map $($i.e., it forgets the metric of positive scalar curvature$)$
		and the last one considers a closed manifold as a~manifold with empty boundary.
	\end{Proposition}
	\begin{proof}
		This is a rather direct consequence of the definitions and well known to the
		experts, at least non-equivariantly, stated, e.g., in \cite[long exact sequence~(4.4)]{Stolz} (but
		without proof). The
		argument for the
		$G$-equivariant case is exactly the same as the classical non-equivariant
		situation. As we are not aware of a treatment available in the journal
		literature, we follow the referee's suggestion and give a complete account
		of the argument here.
		
		First, that the composition of two consecutive maps is zero indeed is a
		direct consequence of the definitions:
		\begin{itemize}\itemsep=0pt
			\item Given $[W,f,g]\in R^{\rm spin}_{n+1}(X)^G$, its image in $\Omega_n(X)^G$ is
			represented by $(\boundary W,f|_{\boundary W})$ which indeed is a boundary, hence represents
			$0$.
			\item Given $[M,f,g]\in \mathrm{Pos}_n^{\rm spin}(X)^G$, its image in $R^{\rm spin}_n(X)^G$ is
			the manifold $M$ with empty boundary. It is bordant in $R^{\rm spin}_n(X)^G$
			to $\varnothing$ and hence represents $0$ via the
			bordism $M\times [0,1]$, where $(M\times \{1\},g)$ is a positive scalar curvature bordism of
			$\varnothing =\boundary M$ to $\varnothing=\boundary\varnothing$.
			\item Given $[M,f]\in \Omega^{\rm spin}_n(X)^G$, its image in
			$\mathrm{Pos}_{n-1}^{\rm spin}(X)^G$ is represented by $\boundary M=\varnothing$, hence
			represents $0$.
		\end{itemize}
		
		For the opposite inclusions of the kernels in the image, the constructions
		are almost as straight-forward:
		\begin{itemize}\itemsep=0pt
			\item Assume that $[M,f]\in \Omega^{\rm spin}_n(X)^G$ is mapped to $0$ in
			$R^{\rm spin}_n(X)^G$. This means that there is a null-bordism $(W,F)$. In the
			case at hand, $\boundary W$ has two disjoint parts: on the one hand~$(M,f)$
			and on the other hand $\bigl(M',f',g'\bigr)$ which is a bordism from $\boundary
			M=\varnothing$ to $\boundary\varnothing=\varnothing$ and which is equipped with
			a metric $g$ of positive scalar curvature. But this means, by definition,
			that $[M,f]=\big[M',f'\big]\in \Omega^{\rm spin}_n(X)^G$ where $M'$ is equipped with a
			metric $g$ of positive scalar curvature. Hence we have the preimage
			\smash{$\big[M',f',g\big]\in \mathrm{Pos}^{\rm spin}_n(X)^G$} of the initial equivariant bordism class
			$(M,f)$.
			\item Assume that $[M,f,g]\in \mathrm{Pos}^{\rm spin}_n(X)^G$ is mapped to $0$ in
			$\Omega^{\rm spin}_n(X)^G$. This means that there is a bordism $(W,F)$ with
			$\boundary(W,F)=(M,f)$. But then $[W,F,g]$ represents a class in~$R^{\rm spin}_{n+1}(X)^G$ which is mapped to $[M,f,g]\in \mathrm{Pos}^{\rm spin}_n(X)^G$.
			\item Finally, assume that \smash{$[W,f,g]\in R^{\rm spin}_{n+1}(X)^G$} is mapped to $0$
			in $\mathrm{Pos}_n^{\rm spin}(X)^G$. This means that there is a second bordism
			$\bigl(W',f',g'\bigr)$ whose boundary is $(\boundary W,f|_{\boundary W},g)$.
			We now consider the closed $G$-manifold $Z:= W\cup_{\boundary W} W'$ with
			map $F:= f\cup_{\boundary W} f'\colon Z\to X$ and the bordism $B:=Z\times
			[0,1]$ with map $F\circ \operatorname{pr}_Z\colon B\to X$. The boundary of $B$ consists
			of three parts: the internal boundary $W'$ between $\boundary W$ and
			$\varnothing =\boundary Z$, equipped with the metric~$g'$ of positive scalar
			curvature extending $g$. The other two boundary parts are $W$ and $Z$, and
			by definition $(B,F\circ \operatorname{pr}_Z, g')$ is a bordism in $R^{\rm spin}_n(X)^G$
			between $(W,f,g)$ and $(Z,f,\varnothing)$ and~${[Z,f]\in
			\Omega^{\rm spin}_n(X)^G}$ represents a preimage of $[W,f,g]\in R^{\rm spin}_n(X)^G$.\hfill $\qed$
		\end{itemize} \renewcommand{\qed}{}
	\end{proof}
	
Let $X$ be a connected $G$-CW-complex with $x_0\in X$ and fundamental group
	$\pi_1(X,x_0)$. Let
	\begin{equation*}
		\pi\colon\ \bigl(\tilde{X},\tilde x_0\bigr)\to (X,x_0)
	\end{equation*}
	be the
	associated universal covering projection. Then the
	proper $G$-action on $X$ lifts to a $\widetilde{G}$-action on $\tilde{X}$,
	where
	\begin{equation*}
		1\to\pi_1(X,x_0)\to \widetilde{G}\to G\to 1
	\end{equation*}
	is an extension of
	discrete groups
	defined as follows: the elements of $\widetilde{G}$ are pairs
	$\bigl(\alpha\colon\tilde X\to\tilde X,g\bigr)$ where $g\in G$ and $\alpha$ covers the
	action map of $g$ on $X$. We define the multiplication by composition of the
	maps $\alpha$ and
	multiplication in $G$.
	
	Note that, by covering theory, because $\tilde X$
	has trivial fundamental group, indeed for each $\tilde x_1\in\tilde X$ with
	$\pi(\tilde x_1)=g\cdot \pi(\tilde x_0)$
	there is a unique lift of the action map sending $\tilde x_0$ to $\tilde
	x_1$.
	
	The projection map $\widetilde{G}\to G$ sends $(\alpha,g)$ to $g$. By
	the above consideration, this map is surjective and its kernel consists of
	the deck transformations which, by covering theory, are identified with
	$\pi_1(X,x_0)$.
	
	\begin{Remark}
		The Stolz exact sequence of Proposition~\ref{prop:Stolz_seq} can always be
		reduced to a sequence where the space $X$ is connected. This is done in two
		steps:
		\begin{enumerate}\itemsep=0pt
			\item[(1)] For every component $X_c$ of $X$ consider the induced $G$-invariant
			subspace $G\cdot X_c$. We then have a disjoint union decomposition into
			$G$-invariant subsets $X=\bigdisjointunion G\cdot X_c$ with~$X_c$ connected
			(choosing one representative in each $G$-orbit of components).
			Clearly, every group in the Stolz exact sequence and the whole sequence
			then split canonically as a~direct product, with one factor for each
			$G\cdot X_c$.
			\item[(2)] For a $G$-space $Y=G\cdot X_c$ with $X_c$ connected, consider the
			subgroup $G_0\subset G$ of all elements which map $X_c$ to itself. Then
			$X_c$ is a $G_0$-space and $Y$ is obtained by ``induction'': $Y=
			G\times_{G_0}X_c$. Whenever we have a $G$-map $f\colon M\to Y$ which
			could be part of a~cycle or a bordism for the groups in the Stolz exact
			sequence, then we obtain~${M_0:=f^{-1}(X_c)}$ a union of components of $M$
			on which $G_0$ acts (restricting the action of $G$ on $M$). Then~${f|_{M_0}\colon M_0\to X_c}$ is a $G_0$-equivariant map and we obtain $M$ and
			$f$ by induction: $M=G\times_{G_0} M_0$ and $f=\id_G\times_{G_0}
			f|_{M_0}$.
			It follows that induction gives an isomorphism (already on the level of
			cycles and relations)
			\begin{equation*}
				\mathrm{ind}_{G_0}^G\colon\ R^{\rm spin}_*(X_c)^{G_0}\to R^{\rm spin}_*(Y)^G
			\end{equation*}
			and the same for the other groups in the sequence of Proposition~\ref{prop:Stolz_seq}
			and the maps between them.
		\end{enumerate}
	\end{Remark}
	
	\begin{Remark}
		The isotropy family $\mathcal{I}\bigl(\tilde{X}\bigr)$ of the action of $\widetilde{G}$
		on $\tilde{X}$ is precisely the inverse image of $\mathcal{I}(X)$ under the projection
		\smash{$\widetilde{G}\to G$}.
		
		This again follows from covering theory: if we fix $x\in X$ with lift
		$\tilde x\in\tilde{X}$ and a subgroup $H$ of~$G$ fixing $x$ we have a
		canonical split $H\to \widetilde{G}$ sending $h\in H$ to the unique pair
		$\bigl(\alpha\colon \tilde X\to \tilde{X},h\bigr)$ such that $\alpha(\tilde x)=\tilde
		x$.
	\end{Remark}

	Then we have the following easy identification.
	\begin{Proposition}\label{extension}
		The $G$-equivariant Stolz exact sequence associated to $X$ is isomorphic to
		the $\widetilde{G}$-equivariant Stolz exact sequence associated to
		$\tilde{X}$.
	\end{Proposition}
	\begin{proof}
		The action of $\pi_1(X)\subset \widetilde{G}$ on $\tilde X$ is
		free. Consequently, for every $\widetilde{G}$-map ${f\colon \bar M\to\tilde X}$ the
		action of $\pi_1(X)$ on $\bar M$ is free. We can therefore quotient out this
		action and obtain a~cycle~${M:=\bar M/\pi_1(X)\to X}$ with residual action of
		$G:=\widetilde{G}/\pi_1(X)$.
		Vice versa, given a~$G$-map~${f\colon M\to X}$ we can pull back $\tilde X\to
		X$ along $f$ and obtain a $\pi_1(X)$-covering $\tilde M\to M$ with map
		$\tilde f\colon \tilde M\to \tilde X$. Because $f$ is a $G$-map, it is
		straightforward to pull back
		also the action of $\widetilde{G}$ to an action on $\tilde M$ which covers
		the action of $G$ on $M$ and such that $\tilde f$ is a $\widetilde{G}$-map
		(indeed, the $\widetilde{G}$-action on $\tilde M$ is defined mapping a point
		$(m,\tilde x)\in \tilde M\subset M\times \tilde X$ by
		$(\alpha,g)\in\widetilde{G}$ to $(gm,\alpha(\tilde x))$).
		
		The same construction gives a bijection between $G$-invariant metrics on $M$
		and $\widetilde{G}$-invariant metrics on $\tilde M$, and also works for
		bordisms.
		These constructions are clearly inverse to each other and preserve all
		additional structure and define the required bijections.
	\end{proof}
	
	\begin{Remark}
		For the paper at hand, the transition to the universal covering as in
		Proposition~\ref{extension} is not really relevant. However, in other
		situations this point of view is really fruitful. To our knowledge,
		essentially all information about non-triviality of the groups
		$R^{\rm spin}_n(X)$ and~$\mathrm{Pos}_n^{\rm spin}(X)$ use higher index theory of the Dirac
		operator. A particularly transparent way to do this is to use the
		equivariant Dirac operator on the universal covering. This essentially means
		that one applies the isomorphism between the groups ``downstairs'' for $X$ and the
		$\pi_1(X)$-equivariant groups ``upstairs'' for the universal covering
		$\tilde X$, as we have already mentioned in the introduction.
		
		It
		is pretty obvious that the same observation does hold for non-free actions: even
		if we want to understand symmetric metrics of positive scalar curvature on a
		compact manifold $M$ with smooth action of a finite group $\Gamma$, it will
		be very useful to pass to the covering space which also sees the symmetries
		induced by the non-trivial fundamental group. This strategy has already been
		implemented in \cite[Section 5]{XieYu}
		and \cite{XieYuZeidler}. The first paper constructs the map from the
		equivariant
		Stolz sequence to analysis (in form of the Higson--Roe exact sequence for the
		K-theory of the Roe $C^*$-algebra) in the
		sense of \cite{PiazzaSchick}. The second paper then uses this to explicitly
		distinguish many bordism classes of equivariant positive scalar curvature
		metrics.
	\end{Remark}

	\subsection{Refinements beyond bordism of positive scalar curvature metrics}
	
	The Stolz positive scalar curvature exact sequence of Proposition
	\ref{prop:Stolz_seq} gives important information about the existence and
	classification of metrics of positive scalar curvature. However, by the very
	definition this information is about \emph{bordism} classes and we are of
	course also interested in a fixed given manifold $M$.
	
	Non-equivariantly, the situation here is very satisfactory: for a given
	connected closed spin manifold $M$ with $\dim(M)\ge 5$, if we use $X=B\pi_1M$
	then $M$ itself admits a metric with positive scalar curvature if (and only
	if) the image of $[u\colon M\to B\pi_1(M)]$ in $R_n^{\rm spin}(B\pi_1(M))$
	vanishes. Moreover, in this case $R_{n+1}^{\rm spin}(B\pi_1(M))$ acts freely and
	transitively on the concordance classes of metrics of positive scalar
	curvature on $M$.
	
	Of course, it would be very desirable to extend such results to the
	equivariant case. Unfortunately, the proof in the non-equivariant case uses as
	an important tool ``handle cancellation'' for Morse functions. It is known
	that equivariantly such cancellation is not even true. Therefore, a~general
	treatment of the equivariant case seems not in reach at the moment. Under very
	special conditions on the action, positive existence results can be
	obtained. The strongest results in this direction we are aware of are obtained
	in \cite{Hanke}. There, also the difficulties are discussed when one attempts
	to obtain more general results. In the paper at hand, we focus on the bordism
	context and do not attempt to contribute to obstruction and classification
	results for G-invariant metrics of positive scalar curvature on a fixed
	manifold $M$.
	
\section[Invariance of R-groups under 2-equivalence]{Invariance of $\boldsymbol{R}$-groups under 2-equivalence}
	
	\begin{Theorem}\label{G2conn}
		Let $f\colon X \rightarrow Y$ be a continuous, $2$-connected $G$-map
		between $G$-CW-complexes $($in particular, it induces a bijection
		between components of fixed point sets and an isomorphism of the
		fundamental groups of all
		components of fixed point sets of $X$ and the corresponding ones of
		$Y)$. Then for $*\ge 6$ the
		functorially induced map
$f_{*}\colon R_{n}^{\rm spin}(X)^{G} \rightarrow R_{n}^{\rm spin}(Y)^{G}
$
		is an isomorphism.
	\end{Theorem}
	\begin{proof}
\!{\it Surjectivity.}\! We start by showing the surjectivity of the map
$f_{*}\colon R_{n}^{\rm spin}(X)^{G} \!\rightarrow\! R_{n}^{\rm spin}(Y)^{G}$.
		Let us consider a class $[W,\varphi\colon W \rightarrow Y, g] \in
		R_{n}^{\rm spin}(Y)^{G}$. We want to find a \textit{bordant} cycle whose reference map factors through $f$.
		
Consider W as a bordism between its boundary $\partial W$ and the empty set
		and choose a~$G$-invariant Morse function $\alpha\colon W\to \RR$ on it with
		critical points rearranged as described in~\cite[Theorem 4.8]{Milnor}, namely
		for any critical points $p_{i}$ and $p_{j}$ such that $f(p_{i}) < f(p_{j})$,
		we have that~${\operatorname{Ind}_f(p_{i}) < \operatorname{Ind}_f(p_{j})}$, where $\operatorname{Ind}_f(p)$ denotes the Morse index at $p$ of the function $f$. Notice that we are going to use the
		enhanced version of this result to the equivariant setting, see for instance
		\cite{Mayer} or \cite{Wasserman}.
		
		Then there exists a suitable $t\in \RR$ such that the subset $W_1:=\alpha^{-1}([0,t])\subset W$ consists only of $G$-handles of dimensions 0, 1 and 2. We immediately obtain a decomposition of $W$ as~${W_{1} \cup W_{2}}$ such that $W_{1}$ is a bordism from the empty set to $M_{1}:=\alpha^{-1}(t)$ and $W_{2}$ a bordism from $M_{1}$ to~$\partial W$.
		Of course, $W_{2}$ has only critical points $p_{i}$ with~$\operatorname{Ind}_{\alpha}(p_{i}) \geq 3$. Consider now the function~$-\alpha$: this is a Morse function on $W_{2}$ seen as a bordism from $\partial W$ to $M_{1}$ with same critical points~$p_{i}$ but with indices now given by $\operatorname{Ind}_{-\alpha}(p_{i})=\dim(W)-\operatorname{Ind}_{\alpha}(p_{i})$. These critical points $p_{i}$ are then associated to $G$-equivariant~${(\operatorname{Ind}_{-\alpha}(p_{i})-1)}$-surgeries, hence of codimension~${\operatorname{Index}_{\alpha}(p_{i})+1}$ which is $\geq 3$.
		
		\begin{figure}[h]
			\centering
			\includegraphics[width=45mm]{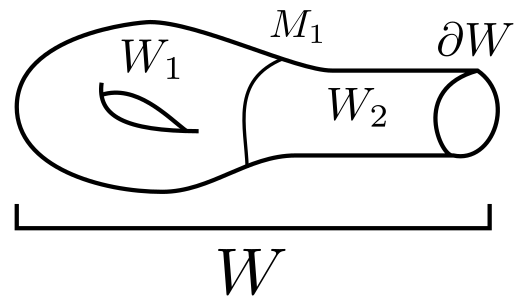}
		\end{figure}
		
		This allows us to apply the Gromov--Lawson theorem in its $G$-equivariant
		version as it is proved in \cite[Theorem 2]{Hanke}. This implies that we can
		extend the metric with positive scalar curvature $g$ on $\partial W$ to a
		$G$-invariant metric with positive scalar curvature $\bar{g}$ on $W_2$. Let us
		denote by $g_{1}$ its restriction to $M_{1}$.
		Observe that the triad $(W_{1},\varphi_{|W_{1}},g_{1})$ defines a class in~\smash{$R_{n}^{\rm spin}(Y)^{G}$} and the manifold $W \times [0,1]$ provides a bordism between \smash{$(W_{1},\varphi_{|W_{1}},g_{1})$} and $(W,\varphi, g)$.
		
		Consider now the natural $G$-equivariant inclusion of the 2-skeleton given by $i\colon
		Y^{(2)}\hookrightarrow Y$. Here we have the following facts:
		\begin{itemize}\itemsep=0pt
			\item 	Since the manifold $W_{1}$ is obtained from the empty set by attaching $G$-handles of dimension~0,~1 and 2, it is homotopy equivalent to a 2-dimensional $G$-CW-complex. It follows from Theorem~\ref{cellular} that the map $\varphi_1:=\varphi_{|W_{1}}$ factors through $i$ up to homotopy.
			\item Since $f$ is 2-connected, up to $G$-homotopy we can assume that its
			restriction to the 2-skeleton
		\smash{$
				f^{(2)}\colon\ X^{(2)} \rightarrow
				Y^{(2)}\rightarrow Y
			$}
			has a right inverse, i.e., there exists a $G$-equivariant map~${h\colon
			Y^{(2)} \rightarrow X^{(2)}}$ such that $f^{(2)}\circ h$ is homotopic to
			$i$. To see this, observe that the existence of such a map $h$ is guaranteed,
			up to $G$-homotopy, by Proposition~\ref{whitehead}. In fact, since $f^{(2)}$
			is $2$-connected and $Y^{(2)}$ has dimension $\leq 2$, it suffices to apply
			Proposition~\ref{whitehead} with $A=Y^{(2)}$, $B=X^{(2)}$, $C=Y$ and $f=f^{(2)}$ to the
			map $i \in [A,C]^G$. The surjectivity of~$f_*$ then implies the existence
			of $h$.
		\end{itemize}
		Thus, we obtain the following commutative diagram of $G$-equivariant maps
$$%\label{diagram1}
			\xymatrix{W_1 \ar[rr]^{\varphi_1} \ar[d]_{\varphi_1}&& Y & \\
				Y^{(2)} \ar[d]_h \ar[urr]^i & & \\
				X^{(2)} \ar[uurr]_{f^{(2)}} \ar[rr]_j && X\ar[uu]_f}
$$
		and, if we set $\psi:= j\circ h\circ \varphi_1\colon W_1\to X$, we obtain by
		construction that
		\[
f_*[W_1,\psi\colon W_1\to X, g_1]=[W,\varphi\colon W\to Y,g]\in R^{\rm spin}_n(Y)^{G},
\]
 which proves that $f_*$ is surjective.
		
\textit{Injectivity.} In order to prove the injectivity of $f_*\colon R_{n}^{\rm spin}(X)^{G}\to R_{n}^{\rm spin}(Y)^{G} $, let us consider a~class $[W,\varphi\colon W \rightarrow X,g]\in R_{n}^{\rm spin}(X)^{G}$ such that its image $f_{*}[W,\varphi\colon W \rightarrow X,g]$ is equal to the trivial element in $R_{n}^{\rm spin}(Y)^{G}$. This means that there exists
		\begin{itemize}\itemsep=0pt
			\item an $(n+1)$-dimensional $G$-manifold with corners $B$, whose codimension 1
			faces are $W$ itself together with a bordism $V$ from $\partial W$ to the empty set, which intersect in the only codimension 2 face $\partial W=W\cap V$;
			\item a $G$-invariant metric $g_V$ on $V$ of positive scalar curvature of
			product type near the boundary which restricts to the metric $g$ on $\partial
			W$;
			\item a $G$-equivariant map $\Psi\colon B\to Y$ which restricts to
			$f\circ\varphi$ on $W$.

		\end{itemize}
		
		\begin{figure}[h]%\label{B}
\vspace*{-6mm}
			\centering
			\includegraphics[width=28mm]{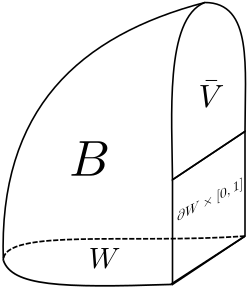}
		\end{figure}
		
		Consider now a $G$-invariant collar neighborhood of $\partial W$ inside $V$ such that the boundary of $B$ is made of three faces of codimension 1: $W$ on the bottom, $\partial W \times [0,1]$ vertically and~${\bar{V}=V\setminus\partial W \times [0,1)}$ on the top.
		
We want to split the bordism $B$, as we did in the proof of surjectivity, into the composition of two bordisms, first from $W$ to a manifold with boundary $W_1$ and then from $W_1$ to $\bar{V}$, such that the first one involves only handle attachments of dimension less or equal than 2 and the second one only of dimension greater or equal than 3.
		Since the vertical boundary face $\partial W \times [0,1]$ is a cylinder, $B$ can be obtained from $W$ by attaching all the handles to the interior of $W$, away from $\partial W \times [0,1]$.
		Hence we can find a Morse function on $B$ which has all critical points there.

Thus, we can decompose $B$ as desired: $B_{1}$ from $W$ to $W_{1}$
		involving only 0, 1, 2 handle attachments and $B_{2}$ from $W_{1}$ to
		$\bar{V}$. We can assume that these two bordisms have vertical
		boundaries faces equal to $\partial W \times [0,1/2]$ and $\partial W
		\times [1/2,1]$, respectively, and therefore that $W_{1}$ has
		boundary equal to $\partial W$.
		
		\begin{figure}[h]%\label{B1-B2}
			\centering
			\includegraphics[width=50mm]{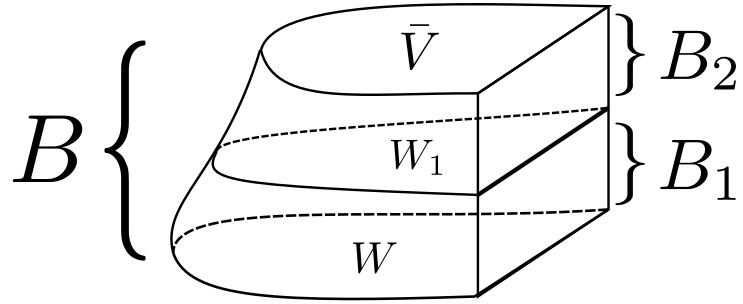}
		\end{figure}

		By construction, the bordism $B_{2}$ is the trace of surgeries of codimension
		$\geq 3$. Therefore, we can apply the equivariant version \cite[Theorem~2]{Hanke} of the Gromov--Lawson theorem to extend the
		metric $g_{\bar{V}}$ to a $G$-invariant metric of positive scalar curvature
		$g_{2}$ on $B_2$. Let us denote by $g_1$ the $G$ -invariant metric of positive
		scalar curvature obtained by restricting $g_{2}$ to~$W_{1}$.
		
The last fact to prove is that $\Psi_{|B_1}\colon B_1\to Y$ factors through $f\colon X\to Y$.
		Indeed, $B_1$ is obtained form $W$ by attaching, up to homotopy, cells
		of dimension up to~2. Define a~map $h\colon B_1\times \{0\}\cup W\times [0,1]\to Y$ as
		the restriction of $\Psi|_{B_1}\circ \mathrm{pr}$ to this subspace of $B_1\times [0,1]$, where~${\mathrm{pr}\colon
		B_1\times [0,1]\to B_1}$ is the obvious projection. Note that by assumption the
		restriction of this map to $W\times \{1\}$ equals $f\circ \varphi$. These are
		exactly the conditions of the relative precursor \cite[Proposition~2.5]{ttd}
		of Proposition~\ref{whitehead} which now implies in particular (because $f$ is
		$2$-connected) that there exists an extension $K\colon B_1\to X$ of
		$\varphi$.
		
		Now observe that $(B_1, \Phi, g_1)$ is a bordism between $(W,\varphi,g)$
		and the trivial cycle $(W_1,K|_{W_1},\allowbreak g_1|_{\partial W})$ in
		$R^{\rm spin}_n(X)^{G}$ (trivial because the metric extends to the metric $g_1$ of
		positive scalar curvature on all of $W_1$), the injectivity of $f_*$ is proved.
	\end{proof}

	\section{Universal spaces}
	
	In algebraic topology, a convenient way to deal with fundamental group
	information is by using the \emph{classifying space} $B\Gamma$ of a (discrete)
	group $\Gamma$. Let us paraphrase its relevant property:
	
	$B\Gamma$ has the following \emph{universal property}: for every connected
	CW-complex $X$ with basepoint~$x_0$ the fundamental group functor gives a
	bijection
	\begin{equation*}
		[(X,x_0), (B\Gamma,y_0)] \to \Hom(\pi_1(X,x_0), \Gamma)
	\end{equation*}
	between the pointed homotopy classes of maps from $X$ to $B\Gamma$ and the group
	homomorphisms from $\pi_1(X)$ and $\Gamma$.
	
	It is known that $B\Gamma$ is characterized by the fact that it is
	connected, has fundamental group~$\Gamma$, whereas all higher homotopy
	groups are trivial.
	
	Our goal now is to achieve the corresponding result for a $G$-CW-complex
	$X$. We observe right away that we have now much richer fundamental group
	information: not only $X$, but also each fixed point set $X^H$ for a subgroup
	$H$ of $G$ (or rather each component of $X^H$) has a fundamental group and
	the action of $g\in G$ as well as fixed point set inclusions induce maps between these fundamental groups.
	
	As a first step, we define the \emph{fundamental groupoid functor} of a $G$-CW-complex as a natural generalization of the fundamental group of a CW-complex.
	In order to do that, let us first recall the definition of fundamental groupoid $\Pi_1(X)$ of a topological space $X$: it is the groupoid whose objects are the points of $X$; whose arrows, from $x$ to $y$ for instance, are equivalence classes of paths starting at $x$ and ending at $y$, where the equivalence relation is given by homotopy of paths with fixed starting and ending points; the composition is given by the concatenation of paths. Observe that the set of arrows starting and ending at a same point $x\in X$ is just the fundamental group of $X$ at $x$.
	\begin{Definition}
		Define for a family of subgroup $\mathcal{F}$ of $G$ the following orbit
		category
		$
		\Orb(G,\mathcal{F})
		$,
		whose objects are all subgroups in $\mathcal{F}$ and morphisms from
		$H\in\mathcal{F}$ to $K\in\mathcal{F}$ are $G$-maps
		$G/H\to G/K$ for $H,K\in \mathcal{F}$. Note\footnote{To be explicit: as
			$G/H$ consists of a single $G$-orbit, the map is determined by the image
			of the coset $H$. Say this image is the coset $gK$. Then by
			$G$-equivariance for each $x\in G$ we must have that $xH$ is mapped to
			$xgK$. To be well defined, for every $x\in G$ and $h\in H$ we require that
			$xhgK$ and $xgK$ are in the same $K$-coset, i.e., $g^{-1}x^{-2}xhg\in K$
			for each $h\in H$, which is the condition $g^{-1}Hg\subset K$.} that any such map is of the form $xH\mapsto xgK$ for a
		well defined coset $[g]\in G/K$, where we also have that $g^{-1}Hg \subset
		K$ ---precisely the condition for this map to be well defined.
		
		As an abbreviation, define $\Orb(G):=\Orb(G,\mathcal{ALL})$ for the family
		of all subgroups of $G$.
	\end{Definition}
	
	\begin{Definition}
		Let $X$ be a $G$-CW-complex. The fundamental groupoid functor of $X$ is the contravariant functor
		\[
		\Pi_1(X;G)\colon\ \Orb(G,\mathcal{I}(X))\to \text{groupoids},
		\]
		which associates
		\begin{itemize}\itemsep=0pt
			\item to a group $H\in \mathcal{I}(X)$ (the isotropy
			family) the fundamental groupoid of $X^H$ restricted
			to the $0$-skeleton of $X^H$ (meaning that we take
			the full subgroupoid), which we denote by~\smash{$\Pi_1\bigl(X^H\bigr)_{|X^H_{(0)}}$},
			\item to a morphism from $H$ to $K$ in
			the category $\Orb(G,\mathcal{I}(X))$ given as $G$-map
			$G/H\mapsto G/K; xH\mapsto xgK$ the morphism of groupoids between
			\smash{$\Pi_1(X^K)_{|X^K_{(0)}}$} and \smash{$\Pi_1\bigl(X^H\bigr)_{|X^H_{(0)}}$} induced by the map
			$X^K\to X^H$ defined as $x\mapsto gx$.
		\end{itemize}
		
		This extends canonically to a functor
		\begin{equation*}
			\Pi_1(X;G)\colon\ \Orb(G)\to \text{groupoids},
		\end{equation*}
		assigning to $H\notin \mathcal{I}(X)$ the empty groupoid. Note that the empty groupoid has a unique
		map to any other groupoid, but is not the target of any map from a non-empty
		groupoid. The former is good because it determines the value of $\Pi_1(X;G)$
		on morphisms in $\Orb(G)$ with domain $H$. The latter is no problem because
		families are closed under conjugation and under taking
		subgroups and therefore if $K\in \mathcal{I}(X)$ and $H\notin\mathcal{I}(X)$
		there never is a morphism from $K$ to $H$ in $\mathcal{ALL}$.\footnote{Recall
			that a morphism from $H$ to $K$ exists if and only if $\exists g\in G$ with
			$g^{-1}Hg\subset K$.}
		
		Below, when dealing with different $G$-spaces with potentially
		different isotropy, we are using this version of the fundamental groupoid functor.
	\end{Definition}
	
	\begin{Remark}
		More conceptually, to construct the fundamental groupoid functor, we can
		observe that we have the canonical morphism $\mathrm{map}^G(G/H, X)\to X^H$ between
		the mapping space of $G$-equivariant maps and the fixed point set, sending
		$f\colon G/H\to X$ to $f(eH)$. The functoriality is then just given by
		precomposition.
	\end{Remark}
	
	Observe that the construction of the fundamental groupoid functor is itself
	functorial. This means that if
	$\varphi\colon Y\to X$ is a $G$-equivariant cellular map between
	$G$-CW-complexes, then there is an induced natural transformation
	$\varphi_{\#}\colon \Pi_1(Y;G)\to \Pi_1(X;G)$
	whose component at $H$ is the homomorphism of groupoids
	\smash{$\varphi_{\#}(H)\colon \Pi_1\bigl(Y^H\bigr)_{|Y^H_{(0)}}\to \Pi_1\bigl(X^H\bigr)_{|X^H_{(0)}}$}
	induced by ${\varphi_{|Y^H}\colon Y^H\to X^H}$.
	
	\subsection{Fundamental groupoid functor realization}
	
	We know that every discrete group is the fundamental group of a 2-dimensional
	CW-complex. The corresponding result holds for our fundamental groupoid
	functors:
	\begin{Proposition}\label{prop:realize_Pi}
		Let $G$ be a discrete group and $\Pi\colon \Orb(G)\to
		Groupoids$ be a functor whose image is given by discrete groupoids. Then there is a $G$-CW-complex $X$
		with $\Pi_1(X;G) \cong \Pi$.
	\end{Proposition}
	\begin{proof}
		The proof is somewhat parallel to the one in the non-equivariant
		case. One has to be careful to work canonically (without choices) as one has
		to achieve compatibility between the different fixed point set data.
		
		To simplify, we make use of some well established constructions (classifying
		space/simplicial set of a small category and its geometric realization), we
		also make use of one very special case of the co-end construction, called
		``tensor product of space valued functors over the orbit category'' in
		\cite[Section~1]{DavisLueck}.
		
		Since the proof is rather long, we summarize here the steps we are going to follow:
		\begin{itemize}\itemsep=0pt
			\item\textbf{Step 1:} we first construct the 0-skeleton $X^{(0)}$ of the candidate $G$-CW-complex $X$;
			\item \textbf{Step 2:} we check that, for each subgroup $H$ of $G$, $\bigl(X^{(0)}\bigr)^{H}$ is in bijection with the units of $\Pi(H)$;
			\item \textbf{Step 3:} we construct the $G$-CW-complex $X$;
			\item \textbf{Step 4:} in order to check that, for each subgroup $H$ in $G$, we have that $\smash{\Pi_1\bigl(X^H\bigr)_{|(X^{(0)})^H}}= \Pi(H)$, we construct an intermediate subspace $Z$ of $X$ which allows to facilitate this computation and we do it;
			\item \textbf{Step 5:} finally, we check that for each morphism $\alpha$ in $\Hom_{\Orb}(H,K)$, the associated map $X^K\to X^H$ induces the groupoid morphism $\Pi(\alpha)$ under the identification of \smash{$\Pi_1\bigl(X^H\bigr)_{|(X^{0})^H}$} with $\Pi(H)$ and the similar one for the subgroup $K$.
		\end{itemize}
		
\textbf{Step 1.} The 0-skeleton of any $G$-CW-complex $X$ with fundamental
		groupoid functor $\Pi$ is directly determined by $\Pi$ itself, more precisely by the units of
		$\Pi(H)$ for the subgroups $H$ of $G$. Considering an orbit $G/H$ as a discrete
		topological space with $G$-action, the ``identity'' functor defines a covariant functor
		$E\colon \Orb(G)\to \GTOP$ sending $G/H$ to the discrete topological $G$-space $G/H$. Then we
		define (and are required to do so)
$X^{(0)}:= E\otimes_{\Orb}\Pi$.
		Concretely, by definition of $\otimes_{\Orb}$ in \cite{DavisLueck}, this is the (discrete) $G$-space obtained as quotient space
		\begin{equation}\label{eq:glue_X0}
			\bigdisjointunion_{H\in \operatorname{Ob}(\Orb(G))} G/H\times \Pi(H)_0 /{\sim},
		\end{equation}
		where the equivalence relation is generated by declaring for all morphisms
		$\alpha\in \Orb(G)$, say $\alpha\colon G/H\to G/K$, that
		\begin{equation*}
			(\alpha(gH), x) \sim (gH,\Pi(\alpha)(x))\qquad\forall x\in \Pi(K)_0.
		\end{equation*}
		Note that we write $\Pi(K)_0$ for the units in the groupoid $\Pi(K)$. The
		$G$-action is induced by the $G$-action on the orbits $G/H$, which is well
		defined because the maps $\alpha$ are $G$-equivariant.
		
\textbf{Step 2.} Let us show that this produces the desired $0$-skeleton of the space $X$ to construct. For this, we have
		to compute the fixed point sets say for the subgroup $H$ of $G$. It is
		straightforward to see that $gK\in G/K$ is fixed by $H$ if and only if
		$g^{-1}Hg\subset K$. At the same time, from such an element $gK$, we then get a well defined $G$-map
$G/H \to G/K$, $uH \mapsto ugK$
		which sends the coset $1H\in G/H$ to $gK\in G/K$, where $1$ denotes the unit in $G$. Thus, all the points~${(gK,x)\in G/K\times \Pi(K)_0}$ in $X^{(0)}$ which are fixed by $H$ are
		identified with a point in the single
		copy $\{1H\}\times \Pi(H)_0$ and therefore the $H$-fixed set of $X^{(0)}$ is
		a quotient of $\Pi(H)_0$. We are done once we have shown that no further
		identifications occur. This follows from the functoriality of $\Pi$:
		whenever two $H$-fixed points are identified with each other, they are also
		identified with a single point in $\{1H\}\times \Pi(H)_0$.
		
\textbf{Step 3.} As a next step, we produce spaces with the correct fundamental groupoids for
		the fixed point sets. For this, we rely on the well established
		construction of the classifying space (as simplicial set) of a small
		category: $|\Pi(H)|$ is the geometric realization of a simplicial set
		associated to the groupoid $\Pi(H)$ with a canonical
		identification \smash{$\Pi_1(|\Pi(H)|)|_{\Pi(H)|^{(0)}} = \Pi(H)$}. In~particular,
		for its zero skeleton we have $|\Pi(H)|^{(0)}=\Pi(H)_0$.
		
		Note, as a remark, that we can not glue together these spaces in the same
		way as we glued together the 0-skeleta in \eqref{eq:glue_X0} to produce the
		desired $G$-space as the gluing process could destroy the fundamental
		groups of the smaller fixed point sets. Instead, we have to carry out a homotopy version of the co-end
		construction which has to the correct $0$-skeleton to obtain the following $G$-space $W$:
		\[%\label{eq:glue_W}
			\Bigg(\bigdisjointunion_{H\in \operatorname{Ob}(\Orb(G))} G/H \times |\Pi(H)|\disjointunion
			\bigdisjointunion_{\alpha\colon G/H\to G/K \atop \in \Hom_{\Orb(G)}(H,K)} G/H\times [0,1] \times |\Pi(K)|\Bigg) \Big{/}{\sim},
		\]
		where the equivalence relation is now generated by a multiple mapping cylinder
		construction, gluing the ends of the spaces associated to morphisms
		appropriately to the spaces associated to the objects. Concretely, this is
		done as follows: let us use the
		notation $(gH,t,x)_\alpha$ to denote an element of the summand $G/H\times
		[0,1]\times |\Pi(K)|$ associated to $\alpha\colon G/H\to G/K\in
		\Hom_{\Orb(G)}$. Then we declare for all morphisms
		$\alpha\in \operatorname{Mor}(\Orb(G))$, say $\alpha\colon G/H\to G/K$
		\begin{alignat}{3}
				&(gH,0, x)_\alpha \sim (gH,\Pi(\alpha)(x)),\qquad&&\forall gH\in G/H,\quad x\in |\Pi(K)|,&\nonumber\\
				&(gH,1,x)_\alpha \sim (\alpha(gH),x), \qquad&&\forall gH\in G/H,\quad x\in |\Pi(K)|.&\label{eq:homotopy_glue}
		\end{alignat}
		
		Note that the $G$-space $W$ contains the $G$-subspace $W_0$ obtained when
		performing this ``homotopy co-end construction'' just to the 0-skeleton
		$|\Pi(\cdot)|^{(0)}=\Pi(\cdot)_0$ of $|\Pi(\cdot)|$, namely{\samepage
		\begin{equation*}
			W_0:= \Bigg(\bigdisjointunion_{H\in \atop \operatorname{Ob}(\Orb(G))} G/H \times \Pi(H)_0\disjointunion
			\bigdisjointunion_{\alpha\colon G/H\to G/K\atop\in \Hom_{\Orb(G)}(H,K)} G/H\times [0,1] \times \Pi(K)_0\Bigg) \Big{/}{\sim}
		\end{equation*}
		with the same equivalence relation as in~\eqref{eq:homotopy_glue}.}
		
		Note that there is the evident projection map $p\colon W_0\to X^{(0)}$ sending
		the class of $(gH,x)$ in~$W$ to the class of $(gH,x)$ in $X^{(0)}$ and sending
		the class of $(gH,t,x)_\alpha$ to the class of $(gH,x)$.
		
		Now we define the desired $G$-CW-complex $X$ by identifying in $W$ all the
		points in $W_0$ with their image points under the surjective map $p$, i.e., as
		the pushout
		\begin{equation*}
			\xymatrix{W_0\ar@{^{(}->}[r]\ar[d]^p & W\ar[d]\\
				X^{(0)}\ar[r]& X.}
		\end{equation*}
		
\textbf{Step 4.} We next have to analyze $X$ and compute in particular $\Pi_1(X;G)$. In order to do that, we define an intermediate $G$-CW-complex $Z$ in the following way.
		Note that the space $W$ contains canonically as subspace
		$\bigdisjointunion_{H\in\Orb(G)} G/H\times |\Pi(H)|$. We therefore obtain as
		a subspace of $X$ the pushout
		\begin{equation*}
			\xymatrix{W_0\ar@{^{(}->}[r]\ar[d]^p & W_0 \cup \bigg(\displaystyle\bigdisjointunion_{H\in\Orb(G)} G/H\times |\Pi(H)|\bigg)\ar[d]\\
				X^{(0)}\ar[r]& Z,}
		\end{equation*}
		where we glue the union of the $G/H\times |\Pi(H)|$ along their $0$-skeleta
		with identifications.
		Observe that we obtain the same space by simply attaching the spaces $|\Pi(H)|$ to the already constructed
		$0$-skeleton $X^{(0)}$ as the following pushout:
		\begin{equation*}
			\xymatrix{
				\displaystyle\bigdisjointunion_{H\in\mathcal{ALL}} G/H \times |\Pi(H)|^{(0)} \ar@{^{(}->}[r]\ar[d]^(.6){p}&
				\displaystyle\bigdisjointunion_{H\in\mathcal{ALL}} G/H \times |\Pi(H)|\ar[d]\\
				X^{(0)} \ar[r]& Z.
			}
		\end{equation*}
		In particular, it is clear that the $0$-skeleton of $X$ is the $0$-skeleton of $Z$,
		i.e., precisely $X^{(0)}$.
		
		Observe that all $1$-cells of $X$ are already contained in the subspace
		$Z$. This space $Z$ is obtained from a disjoint union by identifying along
		$0$-cells. The fundamental groupoids of $X$ and of the
		subspaces $X^H$, for $H$ subgroup of $G$, are then generated by the arrows given by $1$-cells, which are all
		contained in the constituent subspaces of $Z$. By the van Kampen theorem, the
		relations in the fundamental groupoid are then generated precisely by the
		$2$-cells. There are two types of $2$-cells: first the $2$-cells contained in the constituent
		subspaces of $Z$, giving rise to the fundamental groupoids $\Pi(H)$ of
		$|\Pi(H)|$ and second $2$-cells $\{gH\}\times [0,1]\times c_\alpha$ for~${\alpha\colon G/H\to G/K}$ and $1$-cells $c_\alpha$ in $|\Pi(K)|$.
		
		Using these considerations, let us now compute for a subgroup $H$ of $G$ the
		fundamental groupoid \smash{$\Pi_1\bigl( X^H\bigr)|_{X^H_{(0)}}$} of the $H$-fixed set. We follow the
		arguments which lead to the identification of
		\begin{equation*}
			\bigl(\Pi_1\bigl(\bigl(X^{(0)}\bigr)^H\bigr)|_{X^{(0)}}\bigr)_0=\bigl(\Pi_1 \bigl(X^H\bigr)|_{ (X^{(0)} )^H}\bigr)_0=\Pi(H)_0.
		\end{equation*}
		
		Indeed, the $H$-fixed set of $Z$ is given as pushout
		\begin{equation*}
			\xymatrix{ (W_0)^H \ar@{^{(}->}[r]\ar[d]^p & (W_0)^H \cup
				\Bigg(\displaystyle\bigdisjointunion_{K\in\Orb(G);\atop \{gK
					\mid
					g^{-1}Hg\subset K\}} gK\times |\Pi(K)|\Bigg)\ar[d]\\
				\bigl(X^{(0)}\bigr)^H \ar[r]& Z^H}
		\end{equation*}
		with one $1$-cell for each $1$-cell $\{gK\}\times c$ in $ \{gK\}\times
		|\Pi(K)|^H$ for each
		subgroup $K$ of $G$ and coset~${gK\in G/K}$ with $g^{-1}Hg\subset K$ and
		$1$-cell $c$ of $|\Pi(K)|$.

		 \textbf{Step 5.} Now, each such $gK$ gives rise to $\alpha\colon G/H\to G/K$; $uH\mapsto ugK$ as
		above which finally gives rise to a $2$-cell in $X^H$ which identifies the
		class of $\{gK\}\times c$ and of $\{1H\}\times \alpha^*(c)$ in~$\Pi_1\bigl( Z^H\bigr)$. As
		before, the other morphisms in $\Orb(G)$ do not add further relations by
		functoriality of~$\alpha$. Next, the $2$-cells in $|\Pi(H)|$ add precisely
		the required further relations which imply that
		\smash{$\Pi_1\bigl(X^H\bigr)|_{X^H_{(0)}}=\Pi(H)$}. Note that the fact that $\alpha^*\colon \Pi_1(|\Pi(K)|)\to \Pi_1(|\Pi(H)|)$ are groupoid
		homomorphisms implies that no further relations will be created.
		We obtain the desired result~$\smash{\Pi_1\bigl(X^H\bigr)|_{X^H_{(0)}}=\Pi(H)}$ with a canonical
		identification.
		
		From the construction, we also obtain that the induced map coming from a morphism
		$\alpha\colon G/H\to G/K$ of the form $uH\mapsto ugK$, given by fixed
		point inclusion and translation by~${g\in G}$ produces on the fundamental
		groupoids of the fixed point sets $X^K$ and $X^H$ exactly the morphism~$\Pi(\alpha)$.
		
		To summarize, we have constructed the $G$-CW-complex $X$ such that $\Pi_1(X;G)=\Pi$ exactly as required.
	\end{proof}
	
\subsection{Construction of a universal space}\label{sec:universal_space}
	
In the following, given a functor $\Pi\colon \Orb(G)\to \operatorname{groupoid}$ (thought of
	as an abstract fundamental groupoid functor) we construct a universal space
	$B\Pi$ such that canonically $\Pi_1(B\Pi;G)=\Pi$, but all possible higher
	homotopy groups vanish. It has the universal property that there is a
	bijection between algebraic maps between fundamental groupoid functors and
	homotopy classes of maps between $G$-spaces, compare Proposition~\ref{universalspace}.
	
We start with the $G$-CW-complex $X$ of Proposition~\ref{prop:realize_Pi} such
	that canonically $\Pi_1(X;G)=\Pi$.
	Inductively on $k\ge 3$ we construct larger $G$-CW-complexes $X_k$
	by attaching $G$-equivariant $k$-cells to $X_{k-1}$ to kill $\pi_{k-1}$. We start
	by setting $X_2:= X$.
	
	We assume as induction hypothesis that for each subgroup $H$ of~$G$ for each
	component of~$X_{k-1}^H$ the $\pi_j$-th homotopy groups vanishes for $2\le
	j<k-2$. Note that for $k=3$ this is an empty condition and hence the induction
	start for $k=3$ is trivially satisfied.
	
	Then, for each subgroup $H$ of $G$ we first attach
	cells $H/H\times D^k$ of dimension $k$ to \smash{$X_{k-1}^H$} in order to make
	$\pi_{k-1}\bigl(X^H\bigr)$
	trivial. To actually remain in the world of $G$-CW-complexes we induce these
	attaching constructions up and attach $G/H\times D^k$ to $G\cdot
	X_{k-1}^H\subset X_{k-1}$. Note that these $G$-$k$-cells affect also other
	fixed point sets, but there will be no change of $\pi_j$ for~${j<k-1}$, and
	$\pi_{k-1}$ can only get smaller, but as we make sure that it is trivial,
	this property will not be affected. After attaching all these $G$-cells we
	therefore get a new $G$-CW-complex $X_k$ containing~$X_{k-1}$ such that~${\pi_j\bigl(X_k^H,x\bigr)=0}$ for all subgroups $H$ of $G$ and for all $2\le j\le
	k-1$ and for all basepoints~${x\in X_k^H}$.
	
	\begin{Definition}
		The union of all $X_k$ (with the colimit topology) is then called $B\Pi$.
	\end{Definition}
	
	It has the following characteristic property.

	\begin{Lemma}
		The $G$-space $B\Pi$ just constructed contains $X$ as a subcomplex and has the
		same $2$-skeleton as $X$, and
		satisfies that $\pi_j\bigl(B\Pi^H,x\bigr)=0$ for all $j\in\NN$ and for all subgroups
		$H$ of $G$ and all $x\in B\Pi^H$.
	\end{Lemma}

	This classifying space
	$B\Pi$ has the following universal property.
	
	\begin{Proposition}\label{universalspace}
		For every $G$-CW-complex $Y$ and for every natural transformation $\Phi\colon
		\Pi_1(Y;\allowbreak G)\to \Pi$, there exists, unique up to $G$-equivariant
		homotopy, a $G$-equivariant cellular map
		$\varphi\colon Y\to B \Pi$
		such that
		$\varphi_{\#}=\Phi$.
	\end{Proposition}

	\begin{proof}
		We begin by defining the map $\varphi$ on the $G$-equivariant $0$-skeleton of $Y$ by putting
		$\varphi_{|Y_{(0)}}:=\smash{\Phi(\{e\})_{|Y_{(0)}}}$.
		Now, we proceed to define $\varphi$ on the $1$-skeleton of $Y$. For a $G$
		$1$-cell
		$c$ of the form $G/H\times [0,1]$ (with isotropy $H$) pick the ordinary
		$1$-cell $c_H:=H/H\times D^1$ contained in $Y^H$. Then $c_H$
		defines (if we choose an orientation of it) an element \smash{$\gamma \in
		\Pi_1\bigl(Y^H\bigr)_{|Y^H_{(0)}}$}. Define now \smash{$\phi\colon c_H\to \bigl(B\Pi^H\bigr)^{(1)}$} such that
		this map represents the element $\Phi(H)(\gamma)$ in
		$\Pi_1\bigl(B\Pi^H\bigr)|_{ (B\Pi^{ (0)} )^H}$. Since $H$ acts trivially on the target, this extends
		uniquely to a $G$-map on the $G$-cell $c$ with values in the
		$1$-skeleton of $B\Pi$. Of course, we could have picked a different base cell
		in the $G$-cell $c$. But because $\Phi$ is a natural transformation, the
		resulting map is independent of this ---up to the choice of the representative in the
		homotopy class of $\Phi(H)(\gamma)$ which had already to be made anyway.
		This defines the map $\phi$ on the $1$-skeleton of $Y$.
		
		To extend $\phi$ to $Y^{(2)}$, let $c$ be a $G$-$2$-cell of the form
		$G/H\times D^2$. Pick the corresponding ordinary $2$-cell $H/H\times
		D^2$ (with isotropy $H$)
		which is contained in $Y^{H}$. Its attaching map~${\psi\colon S^1\to Y^{H}}$
		is obviously contractible in $Y^{H}$ and hence trivial (i.e., a unit) in
		the fundamental groupoid~\smash{$\Pi_1\bigl(B\Pi^H\bigr)|_{ (B\Pi^{ (0)})^H}$} (we conjugate with a path in
		$Y^H$ to the $0$-skeleton to make~$\psi$
		represent an element of \smash{$\Pi_1\bigl(Y^H\bigr)_{B\Pi^H_{(0)}}$}, the triviality does not
		depend on the choice of such a path). Consequently, $\Phi(H)$ being a map of
		groupoids, also the image under $\Phi(H)$ of this element of the fundamental
		groupoid is trivial (a unit). By construction of $\phi$ on the $1$-skeleton,
		this image element is represented by $\phi\circ\psi$ (up to the chosen
		conjugation with a path to the $0$-skeleton), which is hence contractible in
		$B\Pi^H$. Extend $\phi$ over $H/H\times D^2$ using this contraction and then
		extend it $G$-equivariantly to $c=G/H\times D^2$.
		
		Inductively, we then extend $\phi$ over the $k$-skeleta of $Y$. The extension
		property now follows because each attaching map has contractible image by the
		vanishing of all higher homotopy groups of all components of all fixed point
		sets of $B\Pi$ for the various subgroups of $H$.
		
		By the very construction of $\phi$ on the $1$-skeleton, we have
		$\varphi_{\#}=\Phi$, because the morphism sets of
		\smash{$\Pi(H)=\Pi_1\bigl(B\Pi^H\bigr)|_{ (B\Pi^{ (0)})^H}$} are
		generated by the $1$-cells.
		
		The last step of the proof concerns uniqueness. Choose therefore a
		$G$-equivariant map
		$\varphi'\colon Y \to B \Pi$ such that
		$\varphi'_{\#}=\Phi$. We have to show that $\varphi'$ is $G$-homotopic to
		$\varphi$.
		
		It is immediate from the definition that if $\varphi'_{\#}=\varphi_{\#}$, then
		their restrictions to $Y_{(0)}$ are equal. The construction of the desired
		$G$-homotopy is now done inductively over the skeleta and follows the pattern
		in the non-equivariant case and for the construction part, making use of the
		condition that $\phi_\#'=\phi_\#$ for the extension over the $1$-skeleton and
		of the vanishing of higher homotopy groups for the further extensions.
	\end{proof}
	
	\begin{Example}
		Consider the 2-dimensional torus $\mathbb{T}^2=\mathbb{S}^1\times\mathbb{S}^1$ with the following CW-complex structure:
		the first factor is given by two vertices $v_1$ and $v_2$; two edges
		$e_1$ and $e_2$ both with extremes $v_1$ and $v_2$ (this gives the
		first factor); two loops $l_1$ and $l_2$ attached to $v_1$ and $v_2$
		respectively (these give two copies of the second factor); finally two
		2-cells $c_1$ and $c_2$ suitably attached, the first one to $e_1$,
		$l_1$, $e_1^{-1}$, and $l_2$, the second one to $e_2$, $l_1$,
		$e_2^{-1}$, and $l_2$.
		Let then $\ZZ_2$ act on $\mathbb{T}^2$ by flipping the first factor (in
		particular, swapping $e_1$ and $e_2$) and also $c_1$ and $c_2$,
		and then fixing $l_1$ and $l_2$.
		
		\begin{figure}[ht]
			\centering
			\includegraphics[width=47mm]{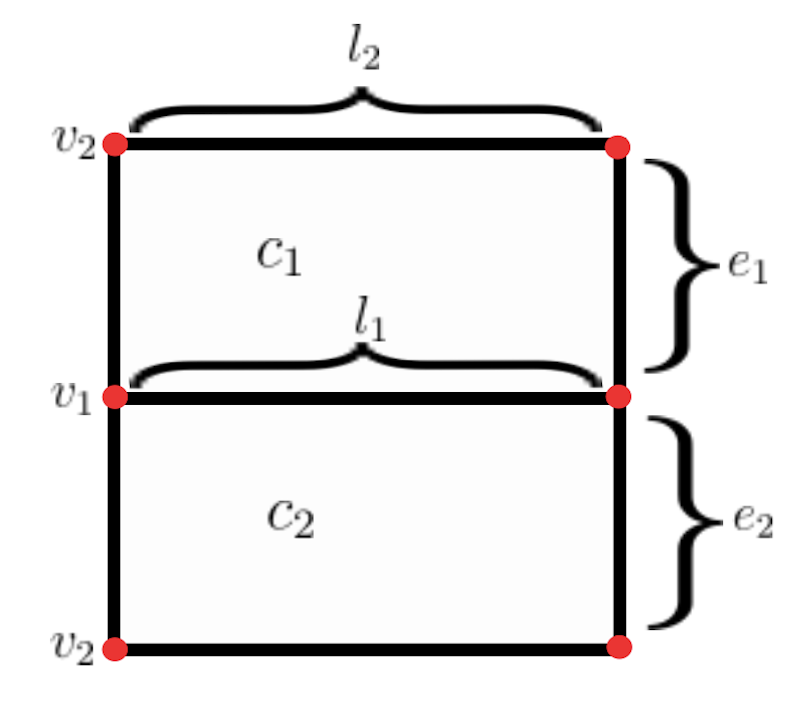}
\vspace{-3mm}

\caption{Torus with action of $\ZZ_2$.}\label{Example}
		\end{figure}
	
In Figure~\ref{Example}, we represent the torus as a square where we identify opposite sides.
		Because the fixed-point sets are the
		aspherical spaces
		$\mathbb{S}^1$ for $\mathbb{Z}_2$ and $\mathbb{T}^2$ for $\{e\}$, this is a
		classifying space for the functor $\Pi_1\bigl(T^2;\ZZ_2\bigr)$, isomorphically given as
	$\Pi\colon \Orb(\ZZ_2)\to \operatorname{Groupoids}$
		which is defined as follows.
		We have to specify which groupoid it assigns to each of the two possible
		$\mathbb{Z}_2$-orbits, namely to the trivial orbit $\ZZ_2/\ZZ_2$
		and the free orbit $\ZZ_2/\{0\}$. Moreover, in the category $\Orb(\ZZ_2)$
		there are only two non-identity morphisms, the collapse map $\ZZ_2/\{0\}\to
		\ZZ_2/\ZZ_2$ and the non-identity bijection $\tau\colon \ZZ_2/\{0\}\to
		\ZZ_2/\{0\}$ and we have to specify the associated morphisms of groupoids.\looseness=-1
		
		Concretely, we define $\Pi$ as follows:
		\begin{itemize}\itemsep=0pt
			\item We assign to $\ZZ_2$ the groupoid $\ZZ\times
			\{v_1,v_2\}$ over $\{v_1,v_2\}$, where $\{S\}$ denotes the trivial groupoid of
			the set $S$ with only identity morphisms.
			\item We assign to $\{0\}\subset \ZZ_2$ the group $\ZZ^2$.
			\item We assign to the collapse map $\ZZ_2/\{0\}\to \ZZ_2/\ZZ_2$ the morphism
			of groupoids which first projects $\ZZ\times \{v_1,v_2\}$ to the factor
			$\ZZ$ and then injects $\ZZ$ into $\ZZ^2$ as the subgroup $\{0\}\times
			\ZZ$. Note that, if we identify $\ZZ^2$ with the fundamental group of the
			torus $T^2$ we constructed above, we think of this as the subgroup generated
			by each of the two loops $l_1$, or $l_2$ (the two loops are homotopic).
			\item We assign to the non-identity map $\tau\colon \ZZ_2/\{0\}\to \ZZ_2/\{0\}$
			the automorphism of $\ZZ^2$ which sends the generator of the first factor
			to its inverse and the second generator to itself.
		\end{itemize}
		
		Note that indeed this is precisely the fundamental groupoid of the
		$\ZZ_2$-space $T^2$ with the chosen CW-structure, and by the contractibility
		of $S^1$ and $T^2$ it is a model for $B\Pi_1\bigl(T^2;\ZZ_2\bigr)$.
		
		It is a nice exercise to carry out the construction of the space $X$ of
		Proposition
		\ref{prop:realize_Pi} for this~$\Pi$. One observes that one attaches quite
		a few $2$-cells and $3$-cells and that way produces a~space~$X$ not homotopy equivalent to $\mathbb{T}^2$. Specifically,
		the $0$-skeleton of $X$ is the $0$-skeleton of~$T^2$. To obtain the
		$2$-skeleton, we use the classifying spaces $T^2$ of $\Pi(\{0\})=\ZZ^2$ and
		$S^1\times \{v_1,v_2\}$ of~${\Pi(\ZZ_2)=\ZZ\times \{v_1,v_2\}}$.
		
		By construction of $X$, we have to take the disjoint union of the $2$-skeleta of
		$\ZZ_2/\{0\}\times T^2$ and of $\ZZ_2/\ZZ_2\times S^1\times \{v_1,v_2\}$ which
		is the disjoint union of two copies of $T^2$ and of $S^1$. To this, we~have to
		glue $4$ cylinders for the $4$ morphisms in $\Orb(\ZZ_2)$
		\begin{gather*}
				 \ZZ_2/\{0\}\times [0,1]\times T^2,\qquad\ZZ_2/\{0\}\times [0,1]\times T^2,
				\qquad \ZZ_2/\{0\}\times [0,1]\times S^1\times \{v_1,v_2\}, \\ [0,1]\times S^1\times
				\{v_1,v_2\}.
		\end{gather*}
		Gluing in the first cylinder, corresponding to $\id_{\ZZ_2/\{0\}}$, produces
		$\ZZ_2/\{0\}\times T^3$. The second cylinder produces another copy of $T^3$,
		glued with the previous two along embedded copies of $T^2$. The third
		cylinder, corresponding to the
		collapse map, homotopically and $\ZZ_2$-equivariantly glues the $2$ copies of
		$S^1$ into this space (without changing the homotopy type). The forth cylinder, corresponding to the identity of $\ZZ_2/\ZZ_2$,
		glues in two more copies of $T^2$ into the space obtained so far along a
		homotopically non-trivial circle in each. From the cellular chain complex we
		can read off that $H^3(X;\ZZ)\cong \ZZ^3$, generated by the
		fundamental classes of the three copies of $T^3$ the first two cylinders
		produced. In particular, this space is not homotopy equivalent to $T^2$.
\end{Example}

	\begin{Remark}\label{rem:BG}
		Observe that when $G=\{e\}$ is the trivial group and $\Pi(\{e\})$ is a
		discrete group $\Gamma$
		(a groupoid with only one object) then $B\Pi$ is a standard
		$CW$-complex with a single $0$-cell. It follows that the space obtained in
		this way is an Eilenberg--Mac~Lane space $B\Gamma=K(\Gamma,1)$.
	\end{Remark}
	
	Now we combine the results obtained in this section and Theorem~\ref{G2conn},
	obtaining as a corollary the fact that the Stolz $G$-equivariant $R$-groups
	depend only on the equivalence class of the fundamental groupoid functor.
	
	Here, we define a natural equivalence between two groupoid valued functors as
	follows.
	\begin{Definition}\label{def:eq_of_fct}
		Let $ \mathcal{C}$ be a small category and $F,G\colon \mathcal{C}\to Groupoid$ be two
		functors. Let~$T$ be a natural transformation between $F$ and $G$. We say
		that $T$ is an \emph{equivalence} if and only if for each object $c$ of $ \mathcal{C}$,
		the functor $T(c)\colon F(c)\to G(c)$ between the groupoids $F(c)$ and~$G(c)$ is an equivalence of groupoids. Recall that the latter condition means that there
		are functors~${S(c)\colon G(c)\to F(c)}$ and the compositions $T(c)\circ S(c)$
		and $S(c)\circ T(c)$ admit natural transformations to the identity functor.
	\end{Definition}
	
	\begin{Remark}\label{rem:grp_equv}
Let $T\colon \mathcal{G}\to \mathcal{H}$ be a morphism of groupoids.
Recall the standard two facts:
		\begin{enumerate}\itemsep=0pt
			\item[(1)] If $T$ is an equivalence then $T$ induces a bijection between the sets of orbits and for
			each unit $x$ of $\mathcal{G}$ an isomorphism of isotropy groups
			\smash{$\mathcal{G}_x^x\to \mathcal{H}_{T(x)}^{T(x)}$}.
			\item[(2)] If $\mathcal{G}$ has a single orbit then for each object $x$ of $\mathcal{G}$ the
			inclusion $\mathcal{G}_x^x\to \mathcal{G}$ induces an equivalence between the isotropy group
			of $x$ in $\mathcal{G}$ (considered as a groupoid with a single object) and the full groupoid $\mathcal{G}$.
		\end{enumerate}
	\end{Remark}
	
	\begin{Lemma}\label{lem:groupoid_eq}
		Let $T\colon \mathcal{G}\to \mathcal{H}$ be a morphism of groupoids such that it induces a
		bijection between the sets of orbits and for each object $x\in \mathcal{G}$ an
		isomorphism of isotropy groups \smash{$\mathcal{G}_x^x\to \mathcal{H}_{T(x)}^{T(x)}$}. Then~$T$
		is an equivalence of groupoids.
	\end{Lemma}
	\begin{proof}
		The assertion follows directly from the following commutative diagram:
		\begin{equation*}
			\xymatrix{\displaystyle\bigsqcup_{[x]\in\pi_0(\mathcal{G})}\mathcal{G}^x_x\ar[r]^{\cong}\ar[d]^(.6){\simeq} &
				\displaystyle\bigsqcup_{[x]\in\pi_0(\mathcal{H})}\mathcal{H}^{T(x)}_{T(x)}\ar[d]^(.6){\simeq}\\
				\mathcal{G}\ar[r]^T&\mathcal{H}.}
		\end{equation*}
		
		Here, we denote $\pi_0(\mathcal{G})$ the set of orbits of $\mathcal{G}$ and we use that $T$
		induces a bijection between~$\pi_0(\mathcal{G})$ and $\pi_0(\mathcal{H})$
		and that a disjoint union of equivalences of groupoids is again an equivalence
		of groupoids.
	\end{proof}
	
	\begin{Lemma}\label{lem:gpd_to_pi1}
		Let $X$ and $Y$ be $G$-CW-complexes and $f\colon X\to Y$ be a cellular
		$G$-map. Then the induced transformation
	$
			f_\# \colon \Pi_1(X;G) \to \Pi_1(Y;G)
	$
		is an equivalence in the sense of Definition~{\rm\ref{def:eq_of_fct}}
		if and only if for each subgroup $H$ of $G$ and for each $x_0\in X^H$
		$f$ induces isomorphisms
		\begin{equation*}
			\pi_j\bigl(f^H\bigr)\colon\ \pi_j(X^H, x_0)\to \pi_j\bigl(Y^H, f(x_0)\bigr) \qquad\text{for} \ j=0,1.
		\end{equation*}
	\end{Lemma}
	\begin{proof}
		First, assume that $f_\#$ is an equivalence. By definition, this means that
		$f^H$ induces an equivalence of groupoids between $\Pi_1\bigl(X^H\bigr)$ and
		$\Pi_1\bigl(Y^H\bigr)$. This, in turn, by Remark~\ref{rem:grp_equv} implies that
		$\pi_j\bigl(f^H\bigr)$ is an isomorphism for $j=0$ and $j=1$.
		The other implication is a special case of Lemma~\ref{lem:groupoid_eq}.
	\end{proof}
	
We now show that the equivariant $R$-groups of $G$-CW-complexes depend only on
	the equivalence class of the fundamental groupoid functor. More precisely, we have the following result.
	
	\begin{Proposition}\label{corol:R_dependence}
		Let $\Pi\colon \Orb(G)\to Groupoids$ be a functor $($thought of as an abstract
		fundamental groupoid functor$)$. Let $X$ be a $G$-CW complex
		and $\Phi\colon \Pi_1(X;G)\to \Pi$ be an equivalence as in Definition~{\rm\ref{def:eq_of_fct}}.
Then for the classifying $G$-map $\varphi\colon X\to B\Pi$ as in Proposition
		{\rm\ref{universalspace}} which induces $\Phi$ on the level of fundamental groupoid
		functors, we have that for $n\ge 6$
		\[%			\label{eq:Stolz_iso}
			\varphi_* \colon R_{n}^{\rm spin}(X)^{G} \rightarrow R_{n}^{\rm spin}(B \Pi)^{G}
		\]
		is an isomorphism.
	\end{Proposition}	
	\begin{proof}
		By Lemma~\ref{lem:gpd_to_pi1}, the map $f\colon X\to Y$ induces isomorphisms
		of $\pi_0$ and $\pi_1$ for each fixed point set $X^H$ and each
		basepoint. Moreover, as all higher homotopy groups of fixed point sets of
		$B\Pi$ are trivial, this map indeed is $2$-connected
		in the sense of Definition~\ref{def:connectedness}.
		Therefore, the assertion is exactly the one of Theorem~\ref{G2conn}.
	\end{proof}
	
	Using this proposition, we arrive at a slight strengthening of Theorem~\ref{G2conn}.

	\begin{Corollary}\label{corol:R_from_pi1}
		Let $f\colon X\to Y$ be a $G$-map between $G$-CW-complexes such that for
		each subgroup $H$ of $G$ the induced map $f^H\colon X^H\to Y^H$ induces an
		isomorphism \[
\pi_j\bigl(f^H\bigr)\colon\ \pi_j\bigl(X^H,x\bigr)\to \pi_j\bigl(Y^H,f(x)\bigr)\]
 for all $x\in
		X^H$ and $j=0,1$. Then for $n\ge 6$ the induced map
		\[
f_*\colon\ R^{\rm spin}_n(X)^G\to R^{\rm spin}_n(Y)^G\]
		is an isomorphism.
	\end{Corollary}
	Note that the assumption on $f$ in the corollary is close to the condition
	to be a $2$-connected, but we do not have any condition on $\pi_2$.
	\begin{proof}
		This improvement relies on the existence of the universal space $B\Pi$. In
		fact, we can post-compose with the classifying map $u\colon Y\to B\Pi_1(Y;G)$
		of Proposition~\ref{universalspace}. Then $u$ and $u\circ f$ both are
		automatically
		$2$-connected and hence both induce isomorphisms between the
		$R$-groups. Then also the third map $f_*$ is an isomorphism.
	\end{proof}
	
\begin{Remark}
		Note that Proposition~\ref{corol:R_dependence} and Corollary
		\ref{corol:R_from_pi1} of course also hold if $G$ is trivial and just state
		that $R_n^{\rm spin}(X)$ depends only on the fundamental group of the CW-complex
		$X$.
	\end{Remark}
	
\subsection*{Acknowledgements}
The authors thank the referees for the careful reading of the paper, improving the presentation and helping to avoid inaccuracies.

\pdfbookmark[1]{References}{ref}
\LastPageEnding

\end{document}